\documentclass[12pt, reqno]{amsart}
\usepackage{amsmath, amsthm, amscd, amsfonts, latexsym, amssymb, graphicx, color}
\usepackage[bookmarksnumbered, colorlinks, plainpages]{hyperref}

\usepackage{amssymb}
\usepackage [all]{xy}
\usepackage{mathtools}
\usepackage{faktor}
\usepackage[percent]{overpic}
\usepackage{mathrsfs}
\usepackage{enumerate}

\DeclareMathOperator{\ch}{ch}

\DeclareMathOperator{\inte}{int}

\DeclareMathOperator{\Rext}{Rext}

\DeclareMathOperator{\Rch}{Rch}
\DeclareMathOperator{\Extrem}{ext}

\DeclareMathOperator{\Atom}{Atom}
\DeclareMathOperator{\Aut}{Aut}

\newcommand{\Lat}{\mathscr{K}}
\newcommand{\RLat}{\mathscr{R}}

\newtheorem{theorem}{Theorem}[section]
\newtheorem{lemma}[theorem]{Lemma}
\newtheorem{proposition}[theorem]{Proposition}
\newtheorem{corollary}[theorem]{Corollary}
\theoremstyle{definition}
\newtheorem{definition}[theorem]{Definition}
\newtheorem{example}[theorem]{Example}
\theoremstyle{remark}
\newtheorem{remark}[theorem]{Remark}

\begin{document}

\setcounter{page}{1}

\title[Convex hull lattices point generated]{Convex hull lattices point generated}
\author[Cardó]{Carles Cardó}
\thanks{{\scriptsize
\hskip -0.4 true cm MSC(2010): 52A10, 06B99, 52B99.
$*$ Under review at \emph{Journal of Discrete \& Computational Geometry}. }}

\begin{abstract}
The simplest way to generate a lattice of convex sets is to consider an initial set of points and draw segments, triangles, and any convex hull from it, then intersect them to obtain new points, and so forth. The result is an infinite lattice for most sets, while only a few initial sets of points perform a finite lattice. By giving an adequate notion of the configuration of points, we identify which sets in the plane define a finite convex hull lattice: four regular families and one sporadic configuration. We explore configurations in the space and higher dimensions.
\end{abstract}

\maketitle

\section {Introduction}\label{Notation}

The class of lattices of convex sets in $\mathbb{R}^d$ is enormous. However, a straightforward way to get one of such lattices is to consider an initial finite set of points, draw segments, triangles and any convex hull, and then intersect them to obtain new points. For most of the initial sets, the generated lattice is infinite. For example, take the vertices of a regular pentagon. By drawing all possible segments, we obtain a pentagon inside the initial pentagon again, and by repeating the operation, we get a third pentagon and so forth. The lattice is finite for only a few configurations of points. We aim to identify such sets of points, first in the plane and later to gain insight into higher dimensions. As far as we know, the literature does not refer to that simple problem close to recreative mathematics, \cite{brass2005research, Gruber1993handbook, toth2017handbook}. Although the answer is simple in the case of the plane, we will need to study a combinatorial version of the convex lattices and perform an adequate point configuration notion.

Let us define the main object of interest for this article.
Recall that the \emph{convex hull} of a set $X \subseteq \mathbb{R}^d$, denoted by $\ch(X)$ is the union of all the possible segments with extremes in $X$.
A set $X$ is said to be a \emph{convex} if $\ch(X)=X$.
The symbol $\ch$ is a \emph{closure operator} \cite{davey2002introduction, Burris}, and hence the set of convex sets in $\mathbb{R}^d$, denoted by $\Lat^d$ is a lattice with operations $X\vee Y= \ch (X \cup Y)$ and $X \wedge Y=X \cap Y$. We will call to abbreviate \emph{ch-lattices} the sublattices of $\Lat^d$.

\begin{definition}\label{DefCHLat} Given a set $X \subseteq \mathbb{R}^d$, we denote by $\Lat^d(X)$ the least sublattice of $\mathscr{K}^d$ containing the set $\{ \{x\} \mid x\in X \}$. We call those lattices \emph{point generated}, and we say that they are \emph{finitely point generated} when $X$ is finite.
\end{definition}
 
Since the intersection and the convex hull of a pair of bounded polytopes are also bounded,  by the principle of structural induction for algebras, elements of $\Lat^d(X)$ are bounded polytopes.
 
Fix the notation and elementary concepts. When necessary, we will abbreviate point singletons as $a=\{a\}$ and the segments as $ab=\{a\}\vee \{b\}=a\vee b$. Given three non-collinear points $a,b,c$ in the plane, $\triangle abc$ denotes the triangle with vertices a,b,c, that is, $\triangle abc=a\vee b \vee c$.
Given a lattice $\mathcal{L}=(L, \vee, \wedge)$ with the order $\leq$, we said that $y \in L$ \emph{covers} $x \in L$ if $x< y$ and for any other $z \in L$ or $x=z$ or $z=y$. If the lattice has a bottom element $0$, we say that $a \in L$ is an \emph{atom} if $a$ covers $0$. We denote the set of atoms by $\Atom(\mathcal{L})$.
A lattice with a bottom element $0$ is \emph{atomic} if, for each non-zero element $x$, there exists an atom $a$ such that $a\leq x$. A lattice is atomistic if it is atomic and every element is a join of some finite set of atoms.
For other elementary concepts on lattice and order theory, we will follow \cite{davey2002introduction} or \cite{Burris}.

We write $\overline{abc}$ to mean that the points $a$,$b$, and $c$ are collinear. A subset $S \subseteq X$ is said \emph{maximal collinear in $X$} if $|S|\geq 3$, all the points in $S$ are collinear and for any $x \in X\setminus S$, the points in $S\cup\{x\}$ are not collinear.
If we want to note that $b$ is between $a$ and $c$, we will write $a{-}b{-}c$, provided $a,b,c$ are different and collinear. A set of points in the plane is said to be \emph{in general position} if it does not contain three collinear points. A set of points $X$ is said to be in \emph{convex position} if for any $x\in X$, $x\not \in \ch(X\setminus \{x\})$. If $X$ is in convex position, it is in general position.
$H^+(ab, c)$ denotes the unique open half-plane delimited by the line that passes by $a$ and $b$ and contains $c$. We define the set $\Theta(a,b,c) = H^+(ab, c) \cap H^+(ac, b) \cap H^-(bc, a)$. If $a,b,c$ are not collinear and $d \in \Theta(a,b,c)$, then $ad \wedge bc\not=\emptyset$.
$\inte (X)$ denotes the interior of a set $X \subseteq \mathbb{R}^d$ and $\partial(X)$ denotes the boundary in the usual Euclidean topological sense. $\Extrem(X)$ denotes the set of extreme points of $X$.
For other elementary terms on convex geometry, we will follow \cite{brondsted2012introduction} or \cite{munkrestopology} for topology.

\section{Configurations and morphisms of configurations}

Any set of a fixed number of collinear points defines the same isomorphic ch-lattice point generated, but this is just a particularity of the one-dimensional space. The goal is to define an equivalence of sets of points that reflects the algebraic properties of its ch-lattice. The literature on discrete geometry offers multiple non-equivalent meanings of point configuration, which are inadequate here. A natural attempt is to consider that two configurations $X,Y$ are equivalent if a mapping $f:\mathbb{R}^d \longrightarrow \mathbb{R}^d$ preserving the convexity such that $f(X)=Y$ exists. These mappings, so-called \emph{rational affine}, have been well studied \cite{artstein2012order}. However, these mappings are too rigid for our proposes. The solution is to consider the relative convex hull \cite{bergman2005lattices}.

\begin{definition} By a \emph{configuration of dimension $d$}, we mean simply a set of points in $\mathbb{R}^d$, for some $d\geq 0$, such that its affine hull has dimension $d$. Given a subset $A \subseteq X$, The \emph{convex hull of $A$ relative to a configuration $X$} is the set
$$\Rch_X(A)=\ch(A) \cap X.$$
When the context gives the configuration $X$, we will abbreviate $\Rch(A)=\Rch_X(A)$. We say that a subset  $A \subseteq X$ is relatively convex (in the context $X$) if $\Rch(A)=A$. 
Given two configurations $X,Y$, a mapping $f:X\longrightarrow Y$ is said a \emph{morphism of configurations} if for any subset $A \subseteq X$
$$f(\Rch_X(A))=\Rch_Y(f(A)).$$
An \emph{isomorphism} $f$ is a bijective morphism. We say that $X$ and $Y$ are \emph{equivalent}, denoted by $X\equiv Y$, if there is an isomorphism between them.
\end{definition}

The composition of morphisms is a morphism of configurations, and the inverse of an isomorphism is also an isomorphism. The terms monomorphism, epimorphism, and automorphism have the usual algebraic sense. The set of automorphisms is denoted by $\Aut(X)$. We call an \emph{abstract configuration} an equivalence class of $\equiv$, but, as it is usual in algebra, by a mild abuse of terminology, we will omit the adjective ``abstract''. Similarly, when a monomorphism exists $X\longrightarrow Y$, we will say that $X$ is a \emph{subconfiguration} of $Y$.

\begin{example} Consider a triangle $\triangle abc$ with a point $d$ in the interior, $T=\{a,b,c,d\}$, and consider a square with vertices $S$. $T$ and $S$ are not equivalent, because $d \in \ch(\{a,b,c\})$, but no triangle in $S$ contains a point.

Now consider three collinear points $C=\{x,y,z\}$, $x{-}y{-}z$. There is an epimorphism of configurations $f:T \longrightarrow C$, given by $f(a)=x$, $f(b)=f(c)=z$, $f(d)=y$.
\end{example}

Contrary to our intuition, the dimension of a configuration, defined as the dimension of its affine hull, is not invariant under the equivalence of configurations. It turns out that a tetrahedron is equivalent to a convex quadrilateral. Moreover, any set of $n$ points in convex position in $\mathbb{R}^d$ is equivalent to the regular $n$-agon in $\mathbb{R}^2$. 
Therefore, we need some invariants to discriminate non-equivalent configurations. Equivalent configurations have the same size. If $f$ is a morphism of configurations, then it preserves betweenness. If $a{-}b{-}c$ in $X$, then $b \in \ch(\{a,c\})$, and $f(a),f(b),f(c)$ are collinear. Therefore, $f(b) \in \ch(\{f(a), f(c)\}$ and $f(a){-}f(b){-}f(c)$. Two useful invariants are the following. First, define the set of \emph{extreme points relative to $X$} as
$$\Rext X=\{ x \in X \mid X\setminus \{x\} \mbox{ is relatively convex}\, \}.$$
It is easily seen that if $X\equiv Y$, $\Rext X\equiv \Rext Y$. The following invariant is, in fact, a family of invariants. Let $Z$ be a class of equivalence of configurations. We count subconfigurations $Z$ in $X$:
$$\#_Z(X)=\mbox{ the number of subconfigurations equivalent to $Z$ in $X$}.$$ 
The first obvious form of classifying finite configurations is by the number of points. Then, we can group them by their sets of relative extreme points. When two configurations have the same set of relative extreme points, we can discern by $\#_z(\cdot)$ for some adequate $Z$. 

If $f:X \longrightarrow Y$ is an isomorphism of configurations and $x\in X$, then $X\setminus\{x\} \equiv Y\setminus \{f(x)\}$. So, configurations of order $n+1$ can be obtained from configurations of order $n$ by adding a point. According to the location of the point, we obtain different configurations. When the number of points grows, listing all configurations is difficult. Fortunately, five points will suffice for the next sections to prove the main theorems.

There is only one configuration with sizes one and two, $L_1$ and $L_2$. See Figure~\ref{DefClasses} for the notations. There are two configurations of size three, $L_3$ and $T_2$. They are not equivalent because $\Rext L_3$ is a segment but $\Rext T_2$ is a triangle. We obtain four-point configurations by adding a point in configurations of size three: $L_4, T_3, I_{0,2}$, and $I_{1,1}$. Except by $T_3$ and $I_{0,2}$, we know that all the pairs of configurations are not equivalent because of their relative extreme points. However, $\#_{T_2}(T_3)=3$ and $\#_{T_2}(I_{0,2})=4$. In the same manner, we obtain the twelve five-point configurations. Using the commented invariants, the reader can check that these are all the configurations. 

The invariant $\#_Z( \cdot)$ satisfies the equalities:
$$\#_{L_3}(X)+ \#_{T_2}(X) ={ |X| \choose 3},$$
provided $|X|\geq 3$, and:
$$\#_{L_4}(X)+ \#_{T_3}(X) +\#_{I_{0,2}}(X)+ \#_{I_{1,1}}(X) ={ |X| \choose 4},$$
if $|x|\geq 4$. More in general, when $|X|\geq k$,
$$ \sum_{Z} \#_Z(X)= { |X| \choose k},$$
where the sum runs over all the configurations $Z$ of size $k$.

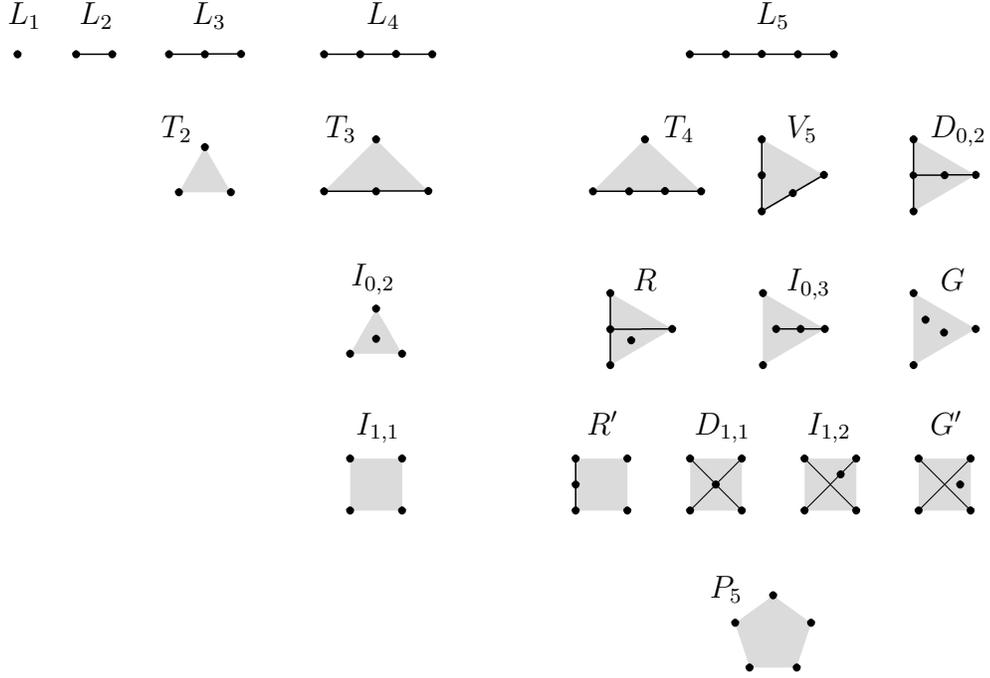
\begin{figure}[tb] 
\vspace{10mm}
\centering
\begin{overpic}[scale=0.15]{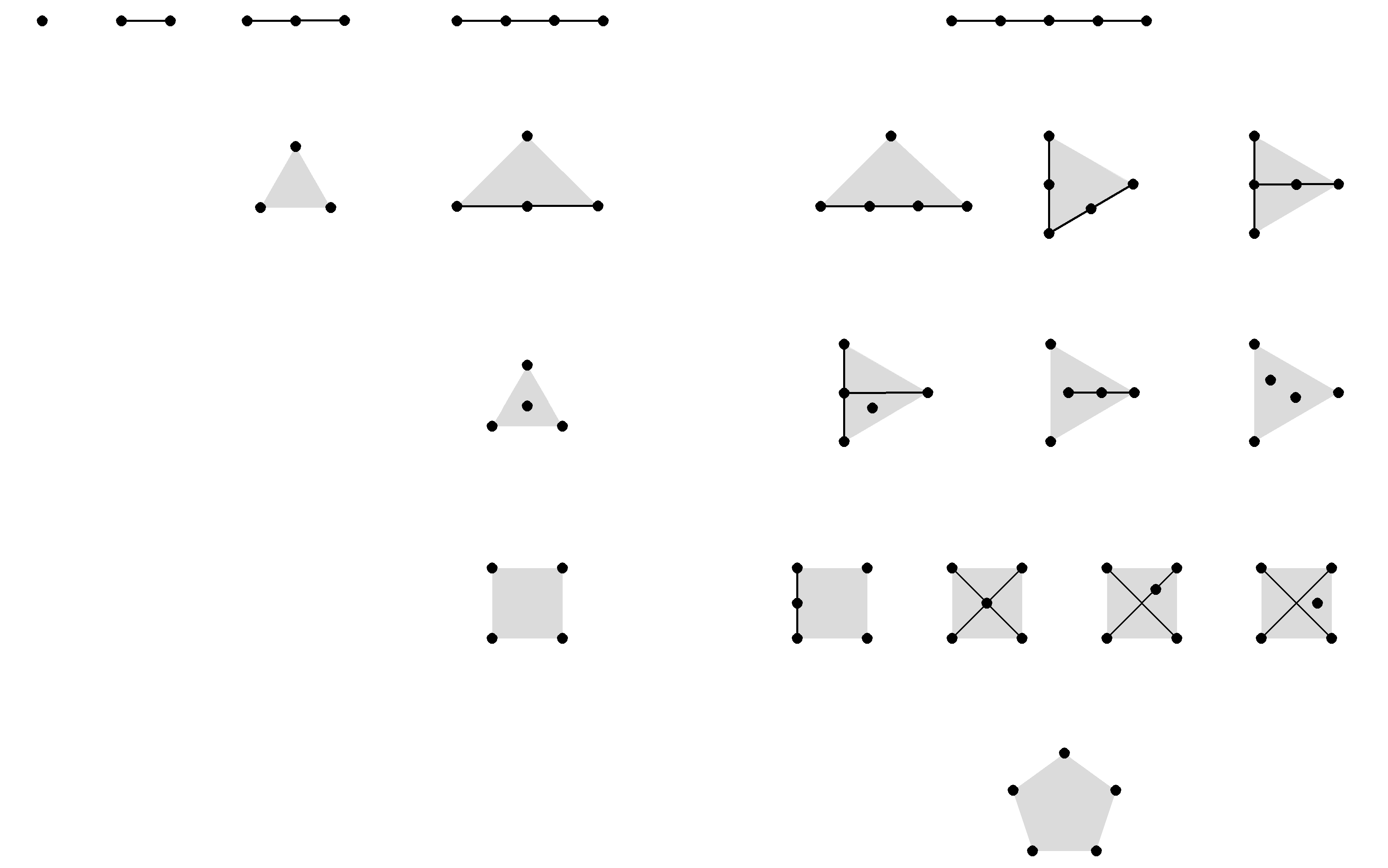}
 \put (2,64) {$L_1$} 
 \put (9,64) {$L_2$} 
 \put (20,64) {$L_3$} 
 \put (37,64) {$L_4$} 
 \put (75,64) {$L_5$} 

  \put (17,53) {$T_2$} 
  \put (33,53) {$T_3$} 
  \put (66,53) {$T_4$} 
  \put (78,53) {$V_5$} 
  \put (92,53) {$D_{0,2}$}

  \put (35.5,38.5) {$I_{0,2}$} 
  \put (63,38) {$R$} 
  \put (78,38) {$I_{0,3}$} 
  \put (93,38) {$G$} 

  \put (36,24) {$I_{1,1}$}

  \put (58.5,24) {$R'$} 
  \put (69,24) {$D_{1,1}$} 
  \put (80,24) {$I_{1,2}$} 
  \put (92,24) {$G'$} 
     
  \put (70.5,8) {$P_5$} 
  
\end{overpic}
\caption{Planar configurations for 1,2,3,4, and 5 points and its designations. See Definition~\ref{DefClasses} for designations $L_{n}$, $T_n$, $D_{p,q}$, and $I_{p,q}$. Shadows indicate the convex hulls. Lines either indicate collinear points or a line where a point cannot lay to preserve the configuration. }\label{Fig0}
\end{figure}

\section{Complete and finitely completable configurations}

\begin{definition} Given a configuration $X$ of dimension $d$, its \emph{completion} is the set of singletons points of $\Lat^d(X)$, and it will be denoted by  $\overline{X}=\{ x \in \mathbb{R}^d \mid \{x\} \in \Lat^d(X) \}$. The set $X$ is said \emph{complete} if $\overline{X}=X$. $X$ is said \emph{finitely completable} if $\overline{X}$ is finite.
\end{definition}

Trivially, if $X$ is finite and complete, it is finitely completable. Notice that $\overline{X}$ is always complete and that $\Lat^d(X)=\Lat^d(\overline{X})$. 
Notice also that if $X \subseteq Y$, then $\Lat^d(X) \subseteq \Lat^d(Y)$. Then, if $X \subseteq Y$, and $Y$ is finitely completable, $X$ is.

\begin{example} If we add the centre to a convex quadrilateral configuration $Q$, we get a non-equivalent configuration $\overline{Q}\equiv D_{1,1}$. We have that $\Lat^2(Q)= \Lat^2(\overline{Q})$. Thus, $Q$ is finitely completable.
\end{example}

\begin{definition} Given a configuration $X$,  $\Rch_X (\cdot)$ is a closure operator, whereby it makes sense to define the lattice $\RLat^d(X)=\{ \Rch_X(Y) \mid Y\subseteq X\}$ with operations $A\vee_R B=\Rch_X(A\cup B)$ and $A\wedge_R B=A \cap B$.
\end{definition}

We will prove that if $X$ is complete, $\RLat^d(X)$ is the combinatoric version of $\Lat^d(X)$, or more precisely, they are isomorphic. We need first a lemma.

\begin{lemma} \label{LemmaExtreme} Let $X$ be complete. For any $A\in \Lat^d(X)$, $\Extrem(A) \subseteq X$. 
\end{lemma}
\begin{proof} To establish the result, we need to know the extreme points of $A \wedge B$ and $A \vee B$ in terms of the extreme points of $A$ and $B$. For the join operation, we have a direct relation: $\Extrem(A \vee B) \subseteq \Extrem (A) \cup \Extrem (B)$.  To calculate the intersection of two polytopes, there are several algorithms from the beginning of the computational geometry; see \cite{preparata2012computational, muller1978finding, chazelle1992optimal}. First, we can triangulate the polytopes into simplices and calculate the intersection of each pair of simplices. Let $S_A$ be a simplex of the triangulation of $A$ and $S_B$ of $B$. Then, to calculate $S_A \wedge S_B$ it suffices  calculate $\partial S_A \cap \partial S_B$, \cite[p. 271]{preparata2012computational}. If $\partial S_A \cap \partial S_B=\emptyset$, either one of the simplices is inside the other, or the intersection is empty. In both cases $\Extrem(S_A \wedge S_B) \subseteq \Extrem (S_A) \cup \Extrem(S_B)$. If the intersection is not empty, then $\Extrem (S_A\wedge S_B) \subseteq \Extrem (\partial S_A \wedge \partial S_B)$.  The boundary of a convex polytope $P$ is conformed by its maximal facets, denoted by $\mathrm{maxF}(P)$. Then, we must only intersect the maximal facets of $S_A$ with the maximal facets of $S_B$. However, even we can improve it. To find the extreme points, we only need to intersect the maximal facets of $S_A$ with edges of $S_B$ and reversely. That is,
\begin{align*}
\Extrem(\partial S_A \wedge \partial S_B) &   \subseteq  \left( \bigcup_{ \substack{F \in \mathrm{maxF}(S_A) \\ x,y \,\in\, \Extrem(S_B) }} xy \wedge F \right)\cup \left( \bigcup_{ \substack{F \in \mathrm{maxF}(S_B) \\ x,y \, \in \, \Extrem(S_A) }} xy \wedge F \right).
\end{align*}  
Look at the first pair of brackets. Since $\Lat^d(X)$ is generated by joins and meets of points in $X$, if we suppose that $\Extrem(S_A), \Extrem(S_B) \subseteq X$, then for each $F \in \mathrm{maxF}(S_A)$, $F \in \Lat^d(X)$, and for each $x,y\in \Extrem(S_B)$, $xy\in \Lat^d(X)$. Since $xy \wedge F$ is a point (and therefore, an atom) and $X$ is complete, $xy \wedge F \in X$. The same holds for the second pair of brackets. Therefore, for each pair of simplices, $\Extrem(S_A \wedge S_B) \subseteq X$. We can sum up all in a single statement:
 $$\Extrem(A), \Extrem(B) \subseteq X \implies \Extrem(A \vee B),\, \Extrem(A \wedge B) \subseteq X.$$
By the structural induction principle, the result follows. 
\end{proof}

\begin{proposition} \label{PropositionIso} If $X$ is complete, $\Lat^d(X) \cong \RLat^d(X)$.
\end{proposition}
\begin{proof} Let $\Phi:  \Lat^d(X) \longrightarrow \RLat^d(X)$, $\Phi(A)=A\cap X$. 
Let us check that $\Phi$ is a bijection and its inverse is $\ch(\cdot)$.  
Clearly, if $A \in \RLat^d(X)$, then $A$ is relatively convex, $\Rch(A)=A$, and then  $\Phi(\ch(A))=\ch(A)\cap X=\Rch(A)=A$. 
Let $B \in \Lat^d(X)$. 
 By Lemma~\ref{LemmaExtreme} above, $\Extrem (B) \subseteq X$. Clearly $\Extrem (B) \subseteq B$, and then $\Extrem (B) \subseteq B\cap X$. Since $\ch$ is monotone, $\ch(\Extrem B) \subseteq \ch(B \cap X)$, that is $B\subseteq \ch(B\cap X)$. Next, consider the other direction of the last inclusion. $B\cap X \subseteq B$. Therefore, $\ch(B \cap X) \subseteq \ch(B)$. Since $B$ is a convex set, $\ch(B)=B$, and then $\ch(B \cap X) \subseteq B$. Hence, $\ch ( \Phi (B))=\ch ( B \cap X)=B$. Thus, $\Phi \circ \ch$ is the identity on $\RLat^d(X)$ and $\ch \circ \, \Phi$ is the identity on $\Lat^d(X)$. 

Next, see that $\Phi$ is a morphism of lattices. For the meet operation, it is trivial: 
\begin{align*}
\Phi(A\wedge B)&=\Phi(A \cap B) =A\cap B \cap X \\
&=A \cap X \cap B \cap X=\Phi(A)\cap \Phi(B)\\
&=\Phi(A)\wedge_R \Phi(B). \end{align*}
For the join operation, $\Phi(A \vee B)=\ch(A\cup B)\cap X$. Notice that we cannot write $\ch(A\cup B)\cap X= \Rch(A\cup B)$, because $A$ and $B$ are not subsets of $X$, but of $\mathbb{R}^d$. However, using $\ch ( B \cap X) = B$, proved in the above paragraph,
\begin{align*}
\Phi(A \vee B) &= \ch(A\cup B)\cap X\\
 &= \ch((A\cup B)\cap X) \cap X\\
&=\ch((A\cap X)\cup (B\cap X)) \cap X\\
&=\Rch((A\cap X)\cup (B\cap X))\\
&=\Rch(A\cap X)\vee_R \Rch(B\cap X)\\
&=(\ch(A\cap X) \cap X)\vee_R (\ch(B\cap X)\cap X)\\
&=(A \cap X)\vee_R (B\cap X)=\Phi(A)\vee_R \Phi(B).
\end{align*}
\end{proof}

\begin{proposition} \label{PropAtomistic} We have the following properties.
\begin{enumerate}[(i)]
\item $\Lat^d(X)$ is atomistic. 
\item $\overline{X}=\bigcup \Atom \Lat^d(X)$.
\item The facial lattice of any polytope in $\Lat^d(X)$ is a sublattice. 
\item $\Lat^d(X)$ is finite iff $X$ is finitely completable.
\end{enumerate}
\end{proposition}
\begin{proof} $\Lat^d(X)=\Lat^d(\overline{X})\cong \RLat^d(\overline{X})$. Since $\RLat^d(\overline{X})$ is by definition atomistic, by Proposition~\ref{PropositionIso}, $\Lat^d(X)$ is, and then $\bigcup \Atom \Lat^d(X)=\bigcup \Atom \RLat^d(\overline{X})=\overline{X}$. That proves (i) and (ii). (iii) follows from (i) and (ii). For (iv), the direction $(\Rightarrow)$ follows from the fact that if $\Lat^d(X)$ is finite, then $\RLat^d(\overline{X})$ is finite and then $\overline{X}$ is finite. $(\Leftarrow)$ Since $\Lat^d(X)$ is point atomistic if $\overline{X}$ is finite, then any element in the lattice can be expressed as the convex hull of a finite number of point atoms, therefore $\Lat^d(X)=\Lat^d(\overline{X})$ must be finite. 
\end{proof}

\begin{definition} Given a morphism of configurations $f: X\longrightarrow Y$, there is at most a morphism of lattices $F: \Lat^d(X) \longrightarrow \Lat^d(Y)$, such that $F(\{x\})=\{f(x)\}$ for any $x\in X$. That is an elementary fact of universal algebra  \cite{Burris}. When such $F$ exists, we will call it the \emph{lattice extension of $f$}, or equivalently, we will say that $f$ is the \emph{configuration restriction of $F$}.
\end{definition}

\begin{theorem} \label{TheoEquiv}  Let $X$ and $Y$ be complete configurations. We have that $X\equiv Y$ iff $\Lat^d(X)\cong \Lat^d(Y)$. 
\end{theorem}
\begin{proof} $(\Rightarrow)$ If $f: X \longrightarrow Y$ is an isomorphism of configurations, the set extension defines an isomorphism of the relative lattices $f: \RLat^d(X) \longrightarrow \RLat^d(Y)$. Let $A,B \in \RLat^d(X)$. For the meet operation, since $f$ is injective, $f(A\wedge B)=f(A\cap B)=f(A)\cap f(B)=f(A)\wedge f(B)$. For the join operation $f(A\vee B)=f(\Rch( A \cup B))=\Rch(f(A\cup B))=\Rch(f(A) \cup f(B))=f(A) \vee f(B)$. 
If we assume that $X,Y$ are complete, then $\ch \circ f \circ \Phi$ is an isomorphism from $\Lat^d(X)$ to $\Lat^d(Y)$, where $\Phi(A)=A\cap X$, see proof of Theorem~\ref{PropositionIso}.

$(\Leftarrow)$ We prove that if $F: \Lat^d(X) \longrightarrow \Lat^d(Y)$ is an isomorphism of lattices, then the natural restriction $f:X \longrightarrow Y$ is an isomorphism of configurations.
An epimorphism of lattices is cover-preserving \cite{crapo1967structure}. Then, $F(\Atom(\Lat^d(X))) \subseteq \Atom(\Lat^d(Y))$. Since $F$ is an isomorphism, $F^{-1}(\Atom(\Lat^d(Y))) \subseteq \Atom(\Lat^d(X))$. Therefore, $$F(\Atom(\Lat^d(X)))=\Atom(\Lat^d(Y)).$$
By Proposition~\ref{PropAtomistic}(ii)  $\overline{X}= \bigcup \Atom\Lat^d(X)$ and $\overline{Y}=\bigcup \Atom\Lat^d(Y)$, which means that the natural restriction, defined by the relation $\{f(x)\}=F(\{x\})$, is a bijective mapping $f:\overline{X} \longrightarrow \overline{Y}$. For any $z \in \overline{X}$, we have that $z \in \ch(Z)$ iff $\{z\} \subseteq \bigvee_{z'\in Z} \{z'\}$, with $Z$ finite. Then,
$$F(\{z\}) \subseteq F\left(\bigvee_{z'\in Z} \{z'\}\right)=\bigvee_{z'\in Z} F(\{z'\})=\bigvee_{z'\in Z} \{f(z')\}=\bigvee_{z''\in f(Z)} \{z''\},$$
which is equivalent to $f(z) \in \ch(f(Z))$. Therefore, $f(\Rch_X(Z)) \subseteq \Rch_Y(f(Z))$.  Using the inverse $F^{-1}$ we obtain the symmetric result $f^{-1}(\Rch_Y(f(Z)) \subseteq \Rch_X(Z)$, and from there, the equality.
\end{proof}

\section {Finite ch-lattices point generated in the Euclidean plane}\label{Euclidean plane}

\begin{figure}[tb] 
\vspace{10mm}
\centering
\begin{overpic}[scale=0.130]{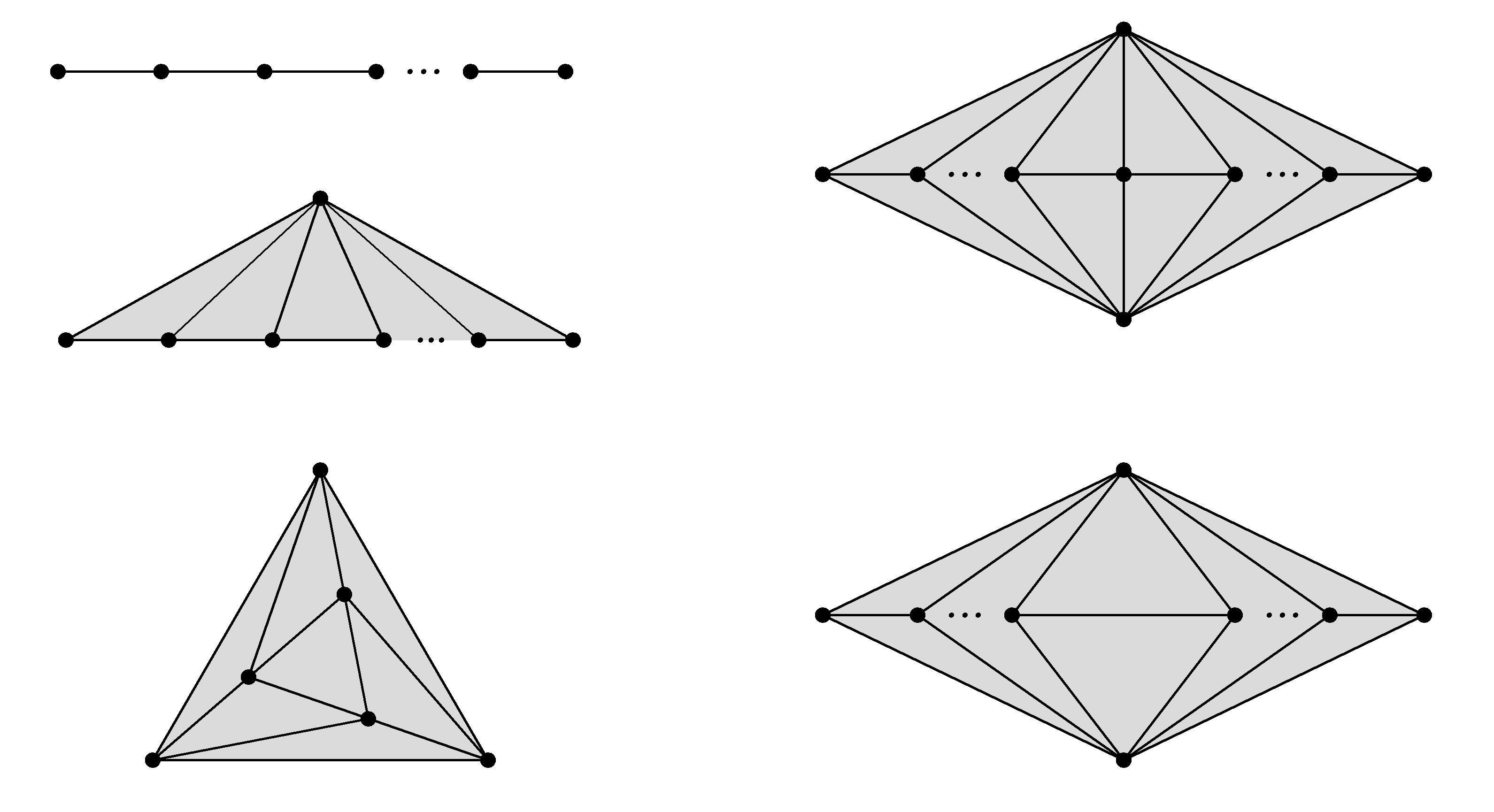}
 \put (7,53) {$L_n$} 

 \put (7,37) {$T_n$} 
 \put (22,42) {$s$}
 
 \put (10,12) {$S_6$}  
 
 \put (55,50) {$D_{p,q}$} 
 \put (76,53.5) {$s$}
 \put (76,40) {$c$}
 \put (76,30.5) {$s'$}
 \put (54,41) {$\underbrace{\qquad \qquad \quad }_{p}$}
 \put (81,41) {$\underbrace{\qquad\qquad \quad }_{q}$}
    
 \put (55,20) {$I_{p,q}$} 
 \put (76,24) {$s$}
 \put (76,0.5) {$s'$}

\end{overpic}
\caption{Configurations from Definition~\ref{DefClasses}.}\label{FigDefClasses}
\end{figure}
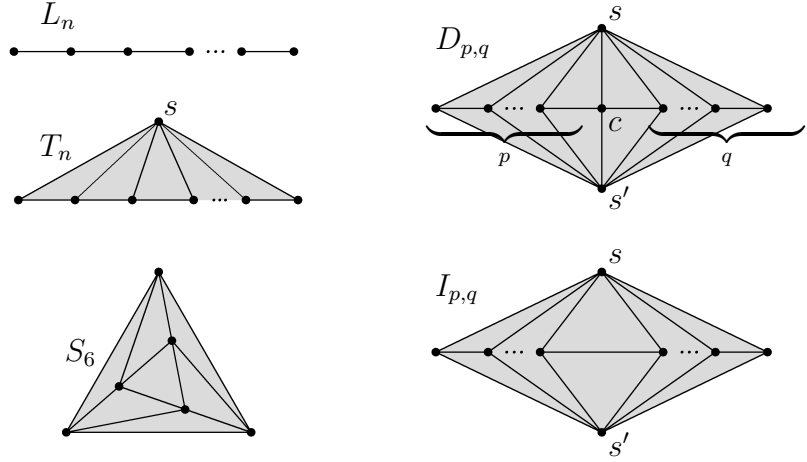
\begin{definition} \label{DefClasses} We define the following four families of planar configurations and an extra configuration.  
\begin{enumerate}[(i)]
\item The \emph{Linear configuration}, $L_n$, consisting in a set of $n>0$ collinear points. 
\item The \emph{Triangular configuration}, $T_n$, consisting in a set of $n>1$ collinear points $a_1{-}\cdots {-}a_n$ and an extra non-collinear point $s$, which we will call \emph{star point}. We will call the \emph{axis} the segment $a_1\vee a_n$. 
\item The \emph{Diamond configuration}, $D_{p,q}$, consisting in a set of $p+q+1$,   collinear points $a_1{-} \cdots {-} a_p {-}c{-} b_1 {-} \cdots {-} b_q$, with $(p,q) \not \in \{(0,0), $ $(1,0), (0,1)\}$, and two \emph{star points}, $s$ and $s'$, non-collinear with the above points such that $ss'\wedge r=c$, where $r= a_1b_q$ is the \emph{axis} of the configuration.  The point $c$ is called the \emph{centre}.
 \item The \emph{Subdiamond configuration}, $I_{p,q}=D_{p,q}\setminus \{c\}$, where $c$ is the center of $D_{p,q}$.
 \item The \emph{Sporadic configuration}, $S_6$, consisting in six points $a,b,c$, $a',b',c'$ such that $a{-}a'{-}c'$, $b{-}b'{-}a'$ and $c{-}c'{-}b'$, and no more collinear subsets. 
\end{enumerate}
\end{definition}

Figure~\ref{FigDefClasses} shows the configurations, the cardinalities of which are: 
$$|L_n|=n, \,\,\,\, |T_n|=n+1, \,\,\,\, |D_{p,q}|=p+q+3, \,\,\,\, |I_{p,q}|=p+q+2, \,\,\,\, |S_6|=6.$$
By simply inspecting Figure~\ref{FigDefClasses}, we see that all configurations are not equivalent.  Notice also that when $p=0$ or $q=0$, $I_{p,q}$ is complete. When $p,q>0$, it is incomplete. 
These configurations suffice to characterise all the finite planar ch-lattices. To streamline our analysis, we will exploit the symmetries of the configurations to reduce the number of cases in the proofs of the next propositions. In particular, $\Aut(D_{1,1})\cong \Aut(I_{1,1})\cong \mathcal{D}_4$,  where $\mathcal{D}_4$ is the dihedral group of order $4$, $\Aut(D_{p,q})\cong \Aut(I_{p,q})\cong \mathbb{Z}/(2)$, for $p,q\not=1$,  and $\Aut(S_6)\cong \mathbb{Z}/(3)$.

\begin{example} \label{ExampleLattices} Let us characterise some of the ch-lattices of the above configurations. Let $\Sigma^*$ be the free monoid of words on the letters $\Sigma$. A word $x$ is said to be a \emph{subword} of $y$ (or that $x$ is a \emph{factor word} of $y$) if there are words $z,t$ such that $y=zxt$. Let $W_n$ be the set of subwords of the word $a_1\cdots a_n$, where $a_1, \ldots, a_n$ are all different letters of $\Sigma$. The order ``be a subword'' on $W_n$ defines a lattice denoted by $\mathcal{W}_n$. Then, it is easy to see that $\Lat^2(L_n) \cong \mathcal{W}_n$. 
In addition, let $\mathcal{C}_2$ be the chain lattice of length two. Then, $\Lat^2(T_n)\cong \mathcal{C}_2 \times \mathcal{W}_n$. 

Figure~\ref{Fig3} shows the Hasse diagrams for the lattices $\Lat^2(L_5)$, $\Lat^2(T_5)$ and $\Lat^2(D_{2,3})\cong \Lat^2(I_{2,3})$. One can infer from the examples the general cases.
\end{example}

\begin{figure}[h] 
\centering
\vspace{7mm}
\begin{overpic}[scale=0.160]{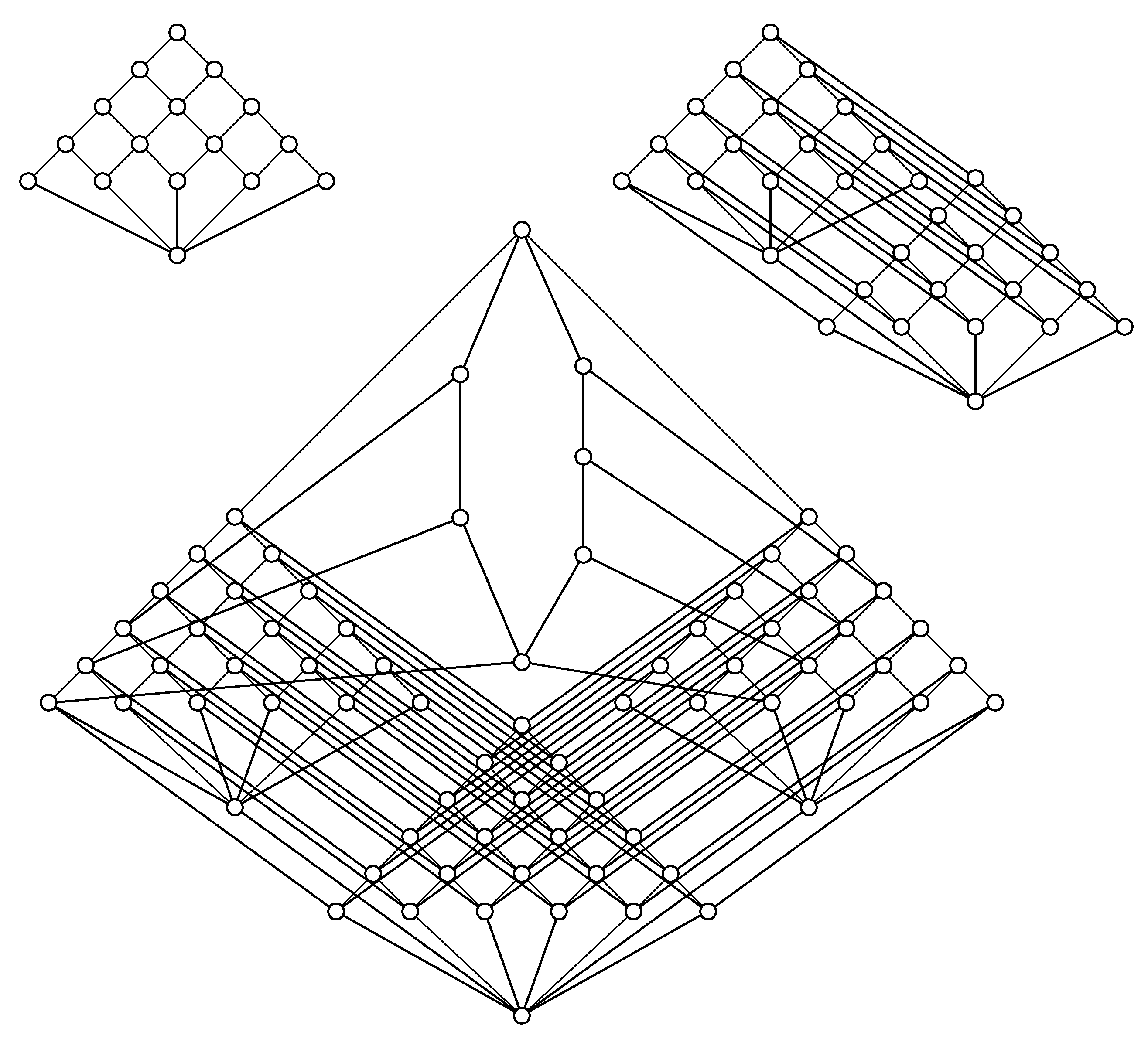}
 \put (8,63) {$\Lat^2(L_5)$} 
 \put (78,50) {$\Lat^2(T_5)$} 
 \put (6,12) {$\Lat^2(D_{2,3})$} 
\end{overpic}
\caption{Hasse diagrams for some planar finite ch-lattices point generated.}\label{Fig3}
\end{figure}

\begin{lemma} \label{LemmaTheo1}  $\Lat^2(V_5)$ is infinite. 
\end{lemma}
\begin{proof} Let $V_5=\{a,b,c,d,e\}$ be a set of points in the plane such that $a{-}b{-}c$ and $a{-}d{-}e$, but  $a,b,c,d,e$ are not collinear. Define the transformation $T(a,b,c,d,e)=(a,b,c,d',e')$ where $e'=ad \wedge be$, and $d'=e'c\wedge bd$, see Figure~\ref{FigLemmas}(a). First, we see that it is well-defined. Since $a,b,d$ are not collinear and $e\in \Theta(a,b,d)$, $ad\wedge eb\not=\emptyset$. Similarly, since $b,c,d$ are not collinear and $e'\in \Theta(b,c,d)$, $e'c\wedge bd\not=\emptyset$. 
$a,b,c,d',e'$ is also a $V_5$ configuration. Clearly $a{-}b{-}c$ and $c{-}d'{-}e'$, by definition. If $a,b,c,d',e'$ were collinear, $a,b,c,d,e$ would be collinear. 
Finally notice that, $e{-}e'{-}b$, with $e'\not=e$ and $e'\not= b$. Applying the transformation as many times as we want, we get
$e{-}e'{-}e''{-}e'''{-}\cdots {-}b$, which means that $\Lat^2(V_5)$ is infinite.  
\end{proof}

\begin{definition} \label{DefClasses2} We say that a configuration in the plane is in \emph{B position} if it has five points $a,b,c,d,e$ such that $a{-}b{-}c$, $e \in H^+(ac,d)$ and that $d,e,b$ are not collinear. 
\end{definition}

\begin{remark} A planar five-point configuration is in B position iff it is equivalent to $V_5$, $R$ or $R'$. 
\end{remark}

\begin{lemma} \label{LemmaTheo2} If $X$ is in B position, then $\Lat^2(X)$ is infinite. 
\end{lemma}
\begin{proof} Let $X=\{a,b,c,d,e\}$ be in B position. We show that we can always find a $V_5$ configuration in $\Lat^2(X)$, which by Lemma~\ref{LemmaTheo1}, implies directly that $\Lat^2(X)$ is infinite since we have a monomorphism of lattices $\Lat^2(V_5) \longrightarrow \Lat^2(X)$. Assume that $ad\wedge ec=\emptyset$, otherwise we are done because $\{a,b,c,ad\wedge ec,d\}\equiv V_5$, see Figure~\ref{FigLemmas}(b). 
That means that either $e \in H^-(dc,a)\cap H^-(ad,c)$ or $d \in H^-(ec,a)\cap H^-(ae,c)$. Thus, $e \in H^-(dc,a)\cap H^-(ad,c)$, and notice that by hypothesis $e,d,b$ cannot be collinear. Then either $ec\wedge ad\not=\emptyset, d$, and then $\{a,b,c, ec\wedge ad, d\}\equiv V_5$, as in Figure~\ref{FigLemmas}(c), or $ec\wedge dc\not=\emptyset, d$, and then $\{a,b,c,ec\wedge dc,d\}\equiv V_5$. The second case $d \in H^-(ec,a)\cap H^-(ae,c)$ runs symmetrically. 
\end{proof}
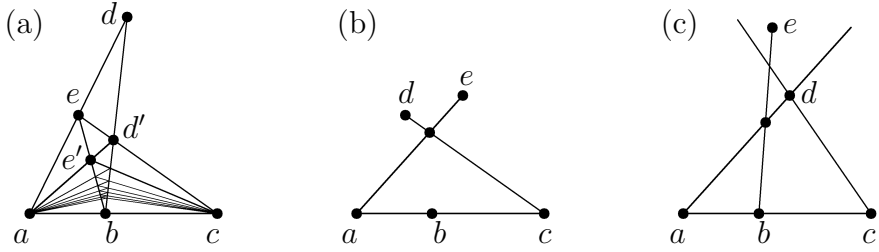
\begin{figure}[tb] 
\centering
\vspace{7mm}
\begin{overpic}[scale=0.13]{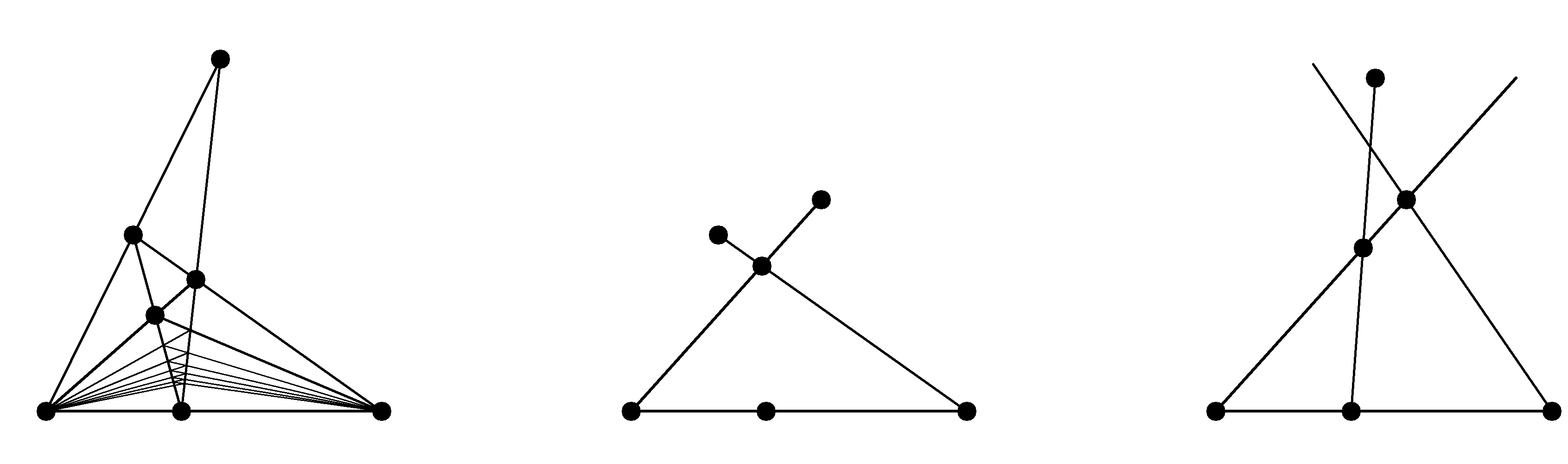}

 \put (0,24) {(a)} 

 \put (1,0) {$a$} 
 \put (11.5,0) {$b$} 
 \put (23,0) {$c$} 
 \put (7,16) {$e$} 
 \put (11,25) {$d$}  
 \put (6.5,8.5) {$e'$} 
  \put (13.5,12) {$d'$} 
 
 \put (38,24) {(b)} 
   
 \put (38.5,0) {$a$} 
 \put (49,0) {$b$} 
 \put (61,0) {$c$} 
  \put (45,16) {$d$} 
  \put (52,18) {$e$} 
   
 \put (75,24) {(c)} 
   
 \put (76,0) {$a$} 
 \put (86,0) {$b$} 
 \put (98,0) {$c$} 
 
   \put (89,24) {$e$} 
  \put (91,16) {$d$}
\end{overpic}
\caption{Schemes for Lemmas~\ref{LemmaTheo1} and \ref{LemmaTheo2}.} \label{FigLemmas}
\end{figure}
\begin{lemma} \label{LemmaTheo3}  If $X$ is a configuration of five points in the plane in general position, then  $\Lat^2(X)$ is infinite. 
\end{lemma}
\begin{proof} Let $X=\{a, b,c,d,e\}$. A classical result asserts that any set of five points in the plane in general position has a subset of four points that form the vertices of a convex quadrilateral, \cite[p. 27]{matouvsek2003introduction}, \cite{erdos1935combinatorial}.
 Let the vertices $a,b,c,d$ be in anticlockwise order of the convex quadrilateral in $X$. The diagonals necessarily intersect, $ac \wedge bd=x \not= \emptyset$. Since points in $P$ are in general position $e,a,c$ are not collinear and $x \not \in ac$. If $e \in H^+(ac, d)$ or $e \in H^-(ac,d)$, then $a,x,c,b,e$ are in B position and $\Lat^2(X)$ is infinite.  \end{proof}

\begin{remark} \label{RemSubcon} Notice that all the configurations from Definition~\ref{DefClasses} are subconfigurations of the diamond configuration, except for the sporadic configuration:
\begin{align*}
&L_{p+q+1} \equiv D_{p,q} \setminus \{s,s'\}, \quad T_{p+q+1} \equiv D_{p,q}\setminus \{s'\}, \quad I_{p,q} \equiv D_{p,q} \setminus \{c\}.
\end{align*}
Any subconfiguration of the diamond configuration is equivalent to only one of the configurations $L_n, T_n, I_{p,q}, D_{p,q}$.
\end{remark}

\begin{theorem} \label{MainTheo} Let $X$ a subset of the plane. The lattice  $\Lat^2(X)$ is finite iff $X$ is a subconfiguration of $D_{p,q}$ for some integers $p,q$, or $X$ is equivalent to $S_6$. 
\end{theorem}

\begin{proof} We only need to prove the necessary condition for which we proceed by induction on the size of $X$, $|X|=n$. Let $n\leq 5$. Then, the statement can be manually checked as follows. Consider Figure~\ref{Fig0}. By Lemma~\ref{LemmaTheo1}, $V_5$ is not finitely completable. $R$ and $R'$ are in B position, and $G, G'$ and $P_5$ are in general position. By Lemmas~\ref{LemmaTheo2} and \ref{LemmaTheo3}, they are not finitely completable. Therefore, the unique configurations finitely completable are $L_1$, $L_2$, $L_3$, $L_4$, $L_5$, $T_2$, $T_3$, $T_4$, $D_{0,2}$, $D_{1,1}$, $I_{0,2}$, $I_{0,3}$, $I_{1,1}$, and $I_{1,2}$. By Remark~\ref{RemSubcon}, all of them are subconfigurations of the diamond configuration.

Assume that the statement is true for any set $X$ with $|X|=n\geq 5$, and let $X'=X \cup \{x\}$ where $x \not \in X$. We distinguish the two cases of the main statement. 
\begin{enumerate}
\item Suppose $X \subseteq D_{p,q}$, with $p+q+3=n$. If we add a point, then $X \cup \{x\} \subseteq D_{p,q} \cup \{x\}$.  If $x \in D_{p,q}$, we are done. Therefore, we can suppose that $x\not \in D_{p,q}$. In this case, if we prove that $D_{p,q} \cup \{x\} \subseteq D_{p',q'}$ or that  $D_{p,q} \cup \{x\} \equiv S_6$, we are done. Notice that in the first case, $X' \subseteq D_{p',q'}$ and $X' \subseteq S_6$ in the second case. However, all the proper configurations of $S_6$ are subconfigurations of $D_{p,q}$. Therefore, $X'$ is a subconfiguration of a diamond configuration or equivalent to the sporadic configuration.  

For brevity we will write $Y=D_{p,q}$ and $Y'=D_{p,q} \cup \{x\}$.  
Let $r$ be the axis and $s,s'$ the star points. We need to distinguish between $n=5$ and $n>5$.
Consider first $n>5$. 
Suppose $x \in H^+(r,s)$. If $x$ and $s$ are not collinear with any point in the axis $r\cap X$, then we have a configuration in B position. By Lemma~\ref{LemmaTheo2}, the ch-lattice is infinite, which is impossible by hypothesis. Suppose that there is some point $a$ in the axis such that $\overline{xsa}$. Since $n> 5$, the axis is conformed at least by four points. We can ignore the point $a$, and then the axis has at least three points and now $s$ and $x$ are not collinear with any point in the axis. Therefore, we have again a configuration in B position. If we suppose that  
$y \in H^-(r,s)$, then by symmetry, it is also impossible. If we want a finite lattice, $y$ must be collinear with $r$. Then, $Y'\equiv D_{p+1,q}$ or $D_{p,q+1}$.

Now let $n=5$. Then, $Y\equiv D_{1,1}$ or $Y\equiv D_{0,2}$. If $Y\equiv D_{1,1}$, $Y$ has symmetry $\Aut(Y)\cong \mathcal{D}_4$ and we can consider that it has two axis. If $x$ is not collinear with any axis, then we have a configuration in B position. So $x$ must be collinear with some axis and then $Y'\equiv D_{1,2}$. 

Consider that $Y \equiv D_{0,2}$, and let $a_1{-}a_2{-}a_3$ the axis. As in the above cases, if $x,s, a_2$ are not collinear, we will obtain a configuration in B position. Thus, $\overline{ysa_2}$.  First, we notice that $y\not \in \triangle a_1 s a_3$ because we could intersect the axis as $x'=sx \wedge r \not=\emptyset$, with which we would obtain a $V_5$ configuration $\{s,x,a_2,a_1,x'\}$. That means that $x$ cannot be between $s$ and $a_2$. And now, we get the sporadic configuration $Y'\equiv S_6$ given by $a_1{-}a_2{-}a_3$, $s{-}a_1{-}s'$, and $x{-}s{-}a_2$. 

\item $X \equiv S_6$. In this case, we will show that we cannot add an extra point $x$ to get a finitely completable configuration. Due to the ternary symmetry of $S_6$, $\Aut(S_6)\cong \mathbb{Z}/(3)$, we only need to check ``one-third'' of the total cases. Let $a,b,c, a',b',c'$ be the vertices of $S_6$, with the collinear sets according to the Definition~\ref{DefClasses}. Notice that $H^+(ac',c) \cup H^+(cb',b) \cup H^+(ba',a)=  \mathbb{R}^2$,
and we only need to check that $\Lat^2(Y)$ is infinite if $x \in H^+(ac',c)$, and the rest of half planes are similar. If $a', c, x$ are not collinear, then we have a configuration in B position. But if $\overline{a'cx}$, then $x'=xc'\wedge cb\not=\emptyset$, and then $x', a,a',c',c$ is also a configuration in B position.   
\end{enumerate}
\end{proof}

\begin{corollary} \label{CoroFirst} Let $X$ a subset of the plane. The lattice $\Lat^2(X)$ is finite iff $X$ is equivalent to $L_n, T_n, D_{p,q}, I_{p,q}$ or $S_6$ for some integers $n,p,q$.
\end{corollary}
\begin{proof} By Theorem~\ref{MainTheo}, $X$ is a subconfiguration of $D_{p,q}$ for some intgers or it is equivalent to $S_6$. If we remove any number of points in $D_{p,q}$, we obtain $L_n, T_n$ or $I_{p,q}$. The reciprocal is trivial since all those configurations are finitely completable.
\end{proof}

Even we have another characterisation of finiteness in the form of forbidden configurations. Let the infinite linear configuration $L_\infty=\{ (n, 0) \in \mathbb{R}^2 \mid n \in \mathbb{N}\}$.

\begin{theorem} \label{TheoSecondCharac} Let $X$ be a complete subset of the plane. The lattice $\Lat^2(X)$ is finite iff $X$ does not contain the $V_5$ and $L_\infty$ subconfigurations.
\end{theorem}
\begin{proof} The direction $(\Rightarrow)$ is trivial from Corollary~\ref{CoroFirst}. For the direction $(\Leftarrow)$, we sketch the argument and leave the details for the reader, which essentially follow the same techniques and lemmas of the above results.
It suffices to see that if we want to build a finite and complete configuration, there cannot appear a $L_\infty$ configuration nor $V_5$ configuration, which implies that neither can appear points in B position nor in general position, because we are assuming that $X$ is complete. The statement can be checked manually if $|X|<5$. If $|X|\geq 5$, we can assume there are at least three collinear points; otherwise, we have five points in general position, leading to an infinite configuration.
Of course, if the lattice is finite, then $L_\infty$ cannot be a subconfiguration. In particular, a greatest maximal collinear set of $X$ is finite. Let $r\geq 3$ be the size of this set. If $r=3$, then it can be proved that $X$ is equivalent to $L_3, T_3, A_3, D_{0,2},D_{1,1}$ or $S_6$. By induction on $r$, suppose that $X$ has a maximal collinear set of size $r$ with no $L_\infty$ nor $V_5$ configurations. Suppose $X'$ with a maximal collinear set of size $r+1$. $X'$ is obtained from $X$ by adding a point in that collinear set. If $V_5$ is a subconfiguration of $X'$, it can be proved that $V_5$ is a subconfiguration of $X$ or that $X$ contains a subconfiguration in B position, which also entails a $V_5$ subconfiguration.
\end{proof}

\section{An insight on the spatial and the general case}

\begin{figure}[tb] 
\centering
\vspace{7mm}
\begin{overpic}[scale=0.18]{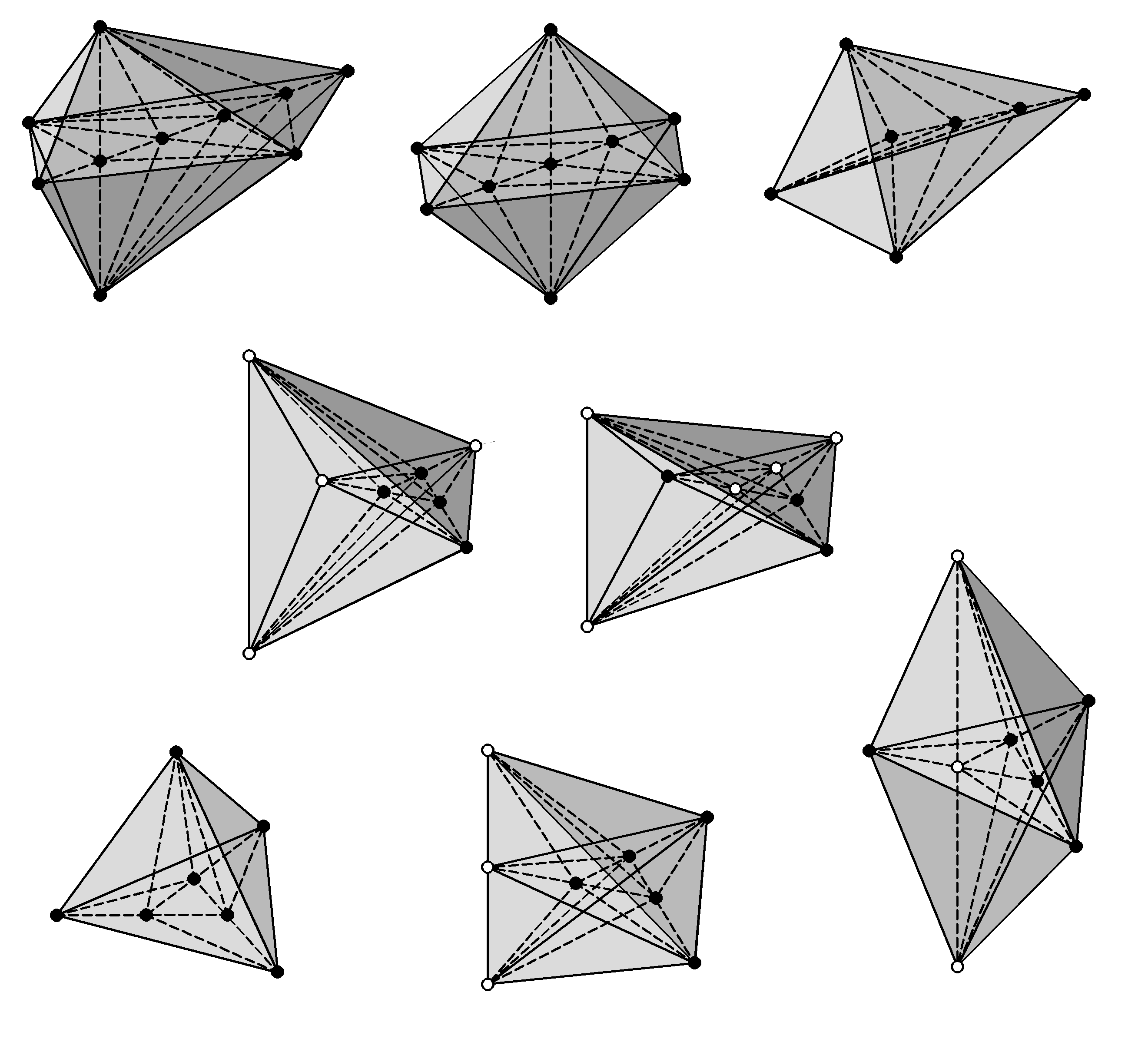}
\put (1,87) {(a)} 
\put (41,87) {(b)} 
\put (83,87) {(c)} 

\put (31,58) {(d)} 
\put (57,58) {(e)} 

\put (5,18) {(f)}
\put (38,18) {(j)} 
\put (73,18) {(g)} 
\end{overpic}
\caption{Some examples of spatial configurations derived from cross-operator and taking subconfigurations: (a) $D_{2,3}*c$, (b) $D_{2,2}*c$, (c) $D_{0,4}*c$. (d), (e), (j), (g) are of the form $S_6*c$, where white dots are collinear or coplanar points. (f) is a subconfiguration of $S_6*c$. }\label{Fig5}
\end{figure}

\begin{figure}[tb] 
\centering
\vspace{7mm}
\begin{overpic}[scale=0.13]{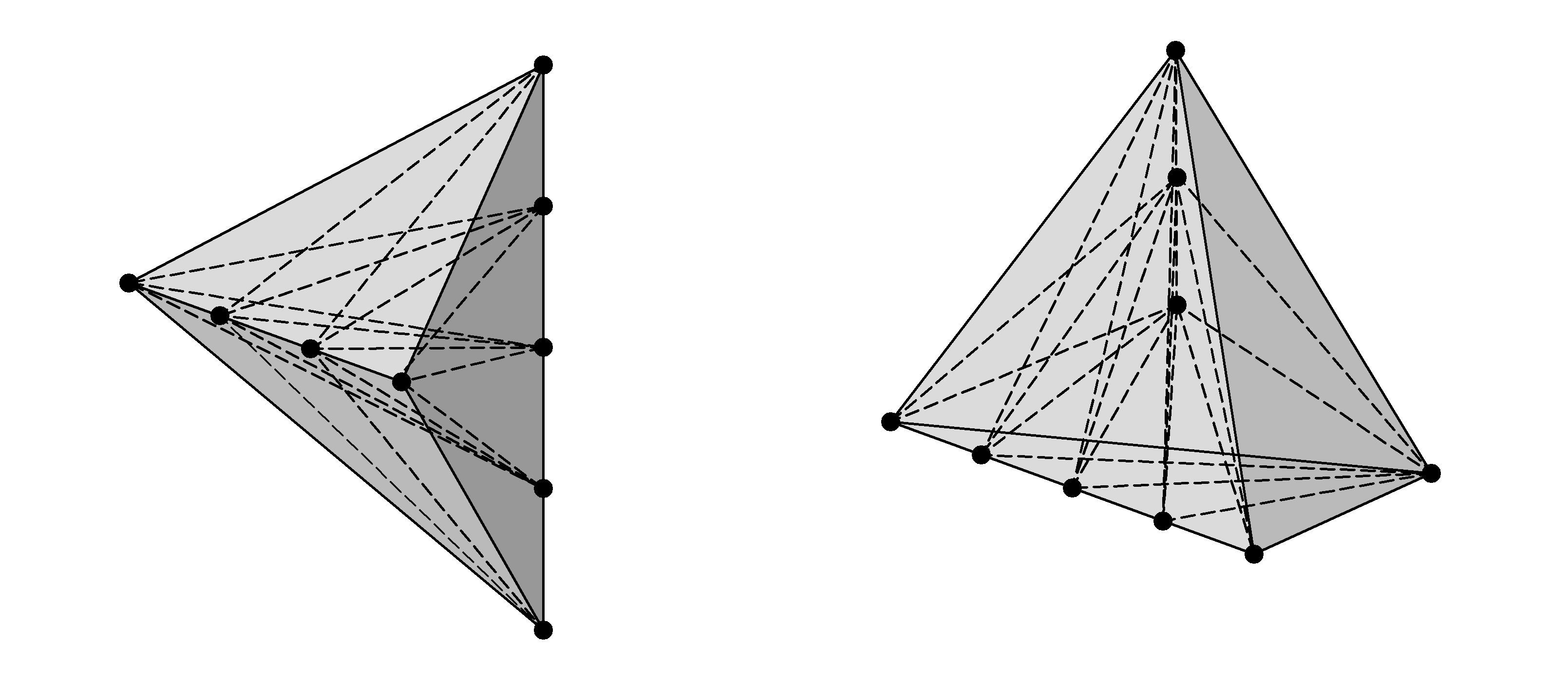}
\put (10,36) {(a)} 
\put (61,36) {(b)} 
\end{overpic}
\caption{Spatial configurations $L_4+L_5$ and $T_5+L_3$.} \label{Fig6}
\end{figure}

The natural continuation tackles the characterisation of finite lattices generated by points in space and the general case in any dimension. For example, the tetrahedron and the octahedron generate finite lattices, while the cube, dodecahedron, and icosahedron generate infinite lattices.

Consider the following operator from configurations in $\mathbb{R}^d$ into configurations in $\mathbb{R}^{d+1}$, which we will call \emph{cross-operator}.
Let $\pi$ be an hyperplane in $\mathbb{R}^{d+1}$ and let $X \subseteq \pi$ a configuration. Let $a$ and $b$ be two points on the opposite sides of $\pi$. Notice that $ab \cap \pi \not=\emptyset$ and that $\dim ab \cap \pi=0$, i.e., the intersection is a point, say $c$. Assume implicitly that $c \in X$ or $c \not \in \ch(X)$, and define $X*c=X\cup \{a,b\}$. Then, the cross-operator is well-defined.
Let $a,a'$ and $b,b'$ be points on the opposite sides of the hyperplane $\pi$, $X$ a configuration in $\pi$ such that $ab \cap \pi= a'b' \cap \pi$. We can define the isomorphism of configurations as $f(a)=a'$, $f(b)=b'$, and $f(x)=x$ for all $x \in X$.
Then, $X\cup \{a,b\} \equiv X \cup \{a',b'\}$, and $X*c$ only depends on the chosen $c$, not on the points $a,b$.

\begin{lemma} $\Lat^d(X)$ is finite iff $\Lat^{d+1}(X*c)$ is.
\end{lemma}
\begin{proof}  $(\Leftarrow)$ Trivial since $\Lat^d(X) \subseteq \Lat^{d+1}(X*c)$. $(\Rightarrow)$ We can construct a surjective mapping $\mathcal{C}_2 \times \mathcal{C}_2 \times \Lat^d(X) \longrightarrow  \Lat^{d+1}(X*c) $, although it is not a lattice epimorphism. Consider the two element lattices $\mathcal{L}_a=\{ \emptyset, \{a\}\}$ and $\mathcal{L}_b=\{ \emptyset, \{b\}\}$, where $ab\wedge \pi=\{c\}$, and $X \subseteq \pi$. Clearly $\mathcal{L}_a\cong \mathcal{L}_b \cong \mathcal{C}_2$. Notice that any element $Z \in \Lat^{d+1}(X*c)$ is of one of the forms:
$$Z=Z', \quad Z=\ch(Z' \cup\{a\}), \quad Z=\ch(Z' \cup \{b\}), \quad Z=\ch(Z' \cup \{a, b\}),$$
for some $Z'\in \Lat^d(X)$. This fact follows from the fact that either $c\in X$ or $c\not \in \ch(X)$. Then, consider the mapping:
\begin{align*}
\mathcal{L}_a \times \mathcal{L}_b \times \Lat^d(X) & \longrightarrow \Lat^{d+1}(X*c)\\
(A,B,C) & \longmapsto \ch(A \cup B \cup C)
\end{align*}
which by the above comment is surjective, and then $|\Lat^d(X*c)|\leq 4 |\Lat^d(X)|$.
\end{proof}

Except for the sporadic configuration $S_6$, any planar configuration is a subconfiguration of $L_n*c$ for some point $c$.
The cross operator and taking subconfiguration are productive forms to obtain higher dimensional configurations. We get many configurations in the space by this process; see some examples in Figure~\ref{Fig5}. However, not all spatial configurations come from planar configurations. For example, let $L_p+L_q$ be the union of $L_{p}$ and $L_{q}$, such that these linear configurations lie over two non-intersecting and non-parallel lines. The intersection of any pair of convex envelopments does not generate a new point, whereby $L_p+L_q$ is finite and complete; Figure~\ref{Fig6}(a). Similarly, denote by $T_p+L_q$ the union of $T_p$ and $L_q$ such that $L_q$ is in one of the semi-spaces delimited by the plane $\pi$ where lies $T_p$, and such that $L_q$ is not parallel to the axis of $T_p$; Figure~\ref{Fig6}(b). Clearly, $L_p+L_q$ is a subconfiguration of $T_p+L_q$: remove the star point of $T_p$. If the family $T_p+L_q$ were the only one not derived from planar configurations, then $\Lat^3(X)$ would be finite iff
\begin{enumerate}[(i)]
\item $X$ is a subconfiguration of $D_{p,q}*c$, or
\item $X$ is a subconfiguration of $S_6*c$, or
\item $X$ is a subconfiguration of $T_p+L_q$
\end{enumerate}
for some integers $p,q$, and some point $c$.
Unfortunately, we have not succeeded in proving it; many cases are under consideration.
We aim that some ideas in this paper open that and other questions, such as how the notion of point configuration is related to other problems in discrete geometry. Some combinatorial enumerative questions arise naturally, such as counting configurations or isomorphy classes of ch-lattices.

There is another generalisation of ch-lattices. Denote by $\Lat_k^d(X)$ the least ch-lattice containing the set of bounded polytopes $X$ of dimension as much $k$. In particular,  $\Lat_0^d(X)=\Lat^d(X)$. For example, given three non-collinear points $a,b,c$, the lattice $\Lat_1^2(\{ab,c\})$ is a sublattice of $\Lat_0^2(\{a,b,c\})$. More in general
$\Lat_k^d (X)$ is always a sublattice of $\Lat_0^d(\Extrem(X))$.  Whereby the finiteness of the latter implies the finiteness of the former. However, if  $V_5=\{a,b,c,d,e\}$, as in Lemma~\ref{LemmaTheo1}, $\Lat_1^2(\{ad, e, bc\})$ is finite, but $\Lat_1^2(\Extrem(\{ad, e, bc\}))=\Lat_0^d(V_5)$ is infinite. A last example from \cite{bergman2005lattices} considers the consecutive vertices of a regular hexagon $H=\{a_1, \ldots, a_6\}$. $\Lat_1^2(\{a_1a_4, a_2a_5, a_3a_6\})$ if infinite.
The problem of discerning when $\Lat_k^d(X)$ is finite seems, at first glance, more challenging than the point generated case, even in the planar case.




\section*{Declarations}

\begin{itemize}
\item Funding. Not applicable
\item Conflict of interest/Competing interests. No conflict of interest exists.
\item Ethics approval and consent to participate. Not applicable.
\item Consent for publication. Not applicable.
\item Data availability. Not applicable. 
\item Materials availability. Not applicable.
\item Code availability. Not applicable. 
\item Author contribution. Not applicable.
\end{itemize}





\bigskip
\bigskip

{\footnotesize {\bf First Author}\; \\ {Departament de Ci\`encies Bàsiques}, {Universitat Internacional de Catalunya, c/ Josep Trueta s/n,} {Sant Cugat del Vallès,  08195, Spain.}\\
{\tt Email: ccardo@uic.es}\\

\end{document}